\documentclass{article}

\usepackage{amssymb}
\usepackage{amsthm}
\usepackage{amsmath}

\usepackage{calc}
\usepackage[top=2.5cm,bottom=2.3cm,inner=2.3cm,outer=2.3cm,bindingoffset=0.5cm,footskip=0.55cm]{geometry}
\begin{document}

\title{Identifiability of multivariate logistic mixture models}

\author{Z.Q. Shi and T.R. Zheng and J.Q. Han}
\maketitle

\abstract{Mixture models have been widely used in modeling of continuous observations. For the possibility to estimate the parameters of a mixture model consistently on the basis of observations from the mixture, identifiability is a necessary condition. In this study, we give some results on the identifiability of multivariate logistic mixture models.}

\section{Introduction}
\label{intro}

Mixture models have provided a mathematical-based approach to the statistical modeling of data generated from some underlying source. Mixture models based probability density function (pdf) have been used successfully in a number of applications ranging from speaker recognition~\cite{reynolds2000speaker} to bioinformatics~\cite{blekas2005spatially}. Comprehensive surveys of estimation techniques, discussions, and applications are available in the work by~\cite{everitt1981finite,titterington1985statistical,mclachlan1988mixture,lindsay1995mixture,mclachlan1997algorithm,mclachlan2000finite,holzmann2006identifiability}.

The mixture of logistic distributions provides a more robust model to the fitting of normal mixture models, as observations that atypical of a component are given reduced weight in the calculation of its parameters, for logistic distribution has similar shape with the normal distribution but has heavier tails and high kurtosis as compared to the exponentially decaying tails of a Gaussian.

Malik and Abraham~\cite{malik1973multivariate} suggested two families of multivariate logistic distribution which are multivariate analogue of the single-dimensional logistic distribution. In this paper, we prove some properties on the identifiability of multivariate logistic mixtures.

%

Identifiability questions of mixtures must be settled before one can meaningfully discuss the problems of parameter estimation of a mixture model consistently on the basis of observations from the mixture. Identifiability makes sure that the correspondence between distributions and mixture parameter vectors is one to one. Lack of identifiability is a serious problem if we intend to classify future observations into one of several classes from the component distributions. Identifiability of mixtures has been discussed by several authors, including Teicher~\cite{teicher1963identifiability}, Yakowitz and Spragins~\cite{yakowitz1968identifiability}, AL-Hussaini and Ahmad~\cite{al2002identifiability}, Ahmad and AL-Hussaini~\cite{ahmad1982remarks}, Ahmad~\cite{ahmad1988identifiability}, and Holzmann et al.~\cite{holzmann2004identifiability}.

\section{Main results}
\label{main_results}
\subsection{Multivariate logistic mixture models}

Malik suggest two families of multivariate logistic distributions (MLD). We build mixture models on the first family~\cite{malik1973multivariate}. He defined the joint distribution function or cumulative distribution function (cdf) of $\mathbf{X}=(X_1, ..., X_p)^{T}$ for MLD as
 \begin{equation}
L(x_1, ..., x_p)=[1+\sum\nolimits_{k=1}^{p}\text{exp}(-x_k)]^{-1},
\end{equation}
and the density function as
\begin{equation}
\begin{split}
&l(x_1, ..., x_p)
\\&=p!\text{exp}(-\sum\nolimits_{k=1}^{p}x_k)[1+\sum\nolimits_{k=1}^{p}\text{exp}(-x_k)]^{-(p+1)},
\end{split}
\end{equation}
where $-\infty<x_k<\infty,k=1,2,...,p$. In order to build mixtures based on MLD, we introduce center $\mathbf{\mu}$ and scale $\mathbf{\Sigma}$ into the distribution and then the cdf turned into
\begin{equation}
L(x_1 ,....,x_p ;\mathbf{\mu} ,\mathbf{\Sigma} ) = [1 + \sum\nolimits_{k = 1}^p {\exp ( - (x_k  - \mu ^k )/\sigma ^k )} ]^{ - 1},
\end{equation}
and the density function turned into
\begin{equation}
\begin{split}
&l(x_1 ,....,x_p ;\mu ,\sum )
\\&= \left( {\prod\nolimits_{k = 1}^p {\sigma ^k } } \right)^{ - 1} \frac{{p!\exp ( - \sum\nolimits_{k = 1}^p {(x_k  - \mu ^k )/\sigma ^k } )}}{{[1 + \sum\nolimits_{k = 1}^p {\exp ( - (x_k  - \mu ^k )/\sigma ^k )} ]^{(p + 1)} }},
\end{split}
\end{equation}
where $\mathbf{\mu}=(\mu ^1 ,\mu ^2 ,...,\mu ^p)$ and scale $\mathbf{\Sigma}=(\sigma ^1 ,\sigma ^2 ,...,\sigma ^p)$. Here, we use the fact that
\begin{equation}
\int {f(x)}  = \int {\Sigma ^{ - 1} f((x - \mu )/\Sigma )dx}. \nonumber
\end{equation}
Let $s$ be the number of mixture components. Let $\mathbf{X}=(X_1, ..., X_p)^{T}$ be a vector valued random variable distributed  according to the following distribution function:
\begin{equation}
G(X;\mathbf{\eta})=\sum\nolimits_{i=1}^s{\pi _i L(X;\mathbf{\mu} _i, \mathbf{\Sigma} _i)},
\end{equation}
where $\mathbf{\eta}=(\mathbf{\pi}, \mathbf{\mu}, \mathbf{\Sigma} )$,  $\mathbf{\pi}=(\pi _1 ,\pi _2 ,...,\pi _s )$, $0<\pi _i<1$, $\sum\nolimits_{i=1}^s{\pi _i}=1$. For $i = 1,2,...,s$ let $\mathbf{\mu}  = (\mathbf{\mu} _1 ,...,\mathbf{\mu} _s )
$, $\mathbf{\Sigma}  = (\mathbf{\Sigma} _1 ,...,\mathbf{\Sigma} _s )$, and $\mathbf{\mu} _i  = (\mu _i^1 ,...,\mu _i^p )\in R^p$,  $\mathbf{\Sigma} _i  = (\sigma _i^1 ,...,\sigma _i^p ) \in (0,\infty )^p$. Then the parameter space is denoted by $\Gamma  = (0,1)^s  \times (R^p )^s  \times ((0,\infty )^p )^s$. $L( \bullet ;\mathbf{\mu} _i ,\mathbf{\Sigma} _i )$ is the cdf of the multivariate logistic $(\mathbf{\mu} _i ,\mathbf{\Sigma} _i )$-distribution.

\subsection{Identifiability}

Identifiability is a necessary condition for the possibility to estimate the parameters of a mixture model consistently on the basis of observations from the mixture. It makes sure that the correspondence between distributions and mixture parameter vectors is one to one.
Let $ \mathfrak{L}= \{ L(\mathbf{X};\mathbf{\mu} ,\mathbf{\Sigma} ):(\mathbf{\mu} ,\mathbf{\Sigma} ) \in \mathbf{\Theta} \}$ constitute the family of $p$-dimensional logistic distributions indexed by a parameter from the parameter space $\mathbf{\Theta}$, where $\mathbf{\Theta}  = R^p  \times (0,\infty )^p$. Let $ G \in \mathfrak{R}$ be a cdf on $\mathbf{\Theta}$. Then $H(\mathbf{X},G) = \int_\mathbf{\Theta}  {L(\mathbf{X};\mathbf{\mu} ,\mathbf{\Sigma} )dG(\mathbf{\mu} ,\mathbf{\Sigma} )}$ defines a mixture distribution. The distribution is called the  $G$-mixture of $\mathfrak{L}$. In this work we are dealing with finite mixtures, that is to say $\mathfrak{R}$ being the set of distributions on $\mathbf{\Theta}$ with finite support. Consequently, we write $G(\mathbf{\mu} ,\mathbf{\Sigma} )$ for the mixture proportion corresponding to  $L( \bullet ;\mathbf{\mu} _i ,\mathbf{\Sigma} _i )$. If $\mathbf{\Sigma}$ is fixed, then $ \mathfrak{L}$ becomes $ \mathfrak{L}_\mathbf{\Sigma} = \{ L_\mathbf{\Sigma}(\mathbf{X};\mathbf{\mu}):(\mathbf{\mu} ,\mathbf{\Sigma} ) \in \mathbf{\Theta}_\mathbf{\Sigma} \}$, where $\mathbf{\Theta}_\mathbf{\Sigma}  = R^p  \times \mathbf{\Sigma}$. Respectively let $\mathfrak{R}_\mathbf{\Sigma}$ being the set of distributions on $\mathbf{\Theta}_\mathbf{\Sigma}$ with finite support, and then $H(\mathbf{X},G_\mathbf{\Sigma}) = \int_\mathbf{\Theta}  {L_\mathbf{\Sigma}(\mathbf{X};\mathbf{\mu})dG_\mathbf{\Sigma}(\mathbf{\mu}})$ defines a mixture distribution on $ \mathfrak{L}_\mathbf{\Sigma}$, where $ G_\mathbf{\Sigma} \in \mathfrak{R}_\mathbf{\Sigma}$. Following Yakowitz and Spragins~\cite{yakowitz1968identifiability}, the mixture model generated by the family $\mathfrak{L}$ with mixing distribution in $\mathfrak{R}$ is said identifiable if for $G_1 ,G_2  \in \mathfrak{R}:H( \bullet ,G_1 ) = H( \bullet ,G_2 ) \Leftrightarrow G_1  = G_2$.
\newtheorem{thm}{Theorem}
\begin{thm}
\label{thm:main_results_1dim}
When $p=1$, the class $\mathfrak{L}$ with parameter set $\mathbf{\Theta}$ and mixing distribution in $\mathfrak{R}$ as defined above is identifiable.
\end{thm}

\begin{thm}
\label{thm:main_results_pdim}
When $p>1$, the class $\mathfrak{L}_\mathbf{\Sigma}$ with parameter set $\mathbf{\Theta}_\mathbf{\Sigma}$ and mixing distribution in $\mathfrak{R}_\mathbf{\Sigma}$ is identifiable.
\end{thm}

The usual approaches based on moment generating function, characteristic function and Laplace transforms~\cite{teicher1963identifiability, yakowitz1968identifiability, al2002identifiability} can not be applied to this situation.  Individual methods are needed. We first prove two lemmas which will be used in the proof of our theorems.

\newtheorem{lem}{Lemma}
\begin{lem}
\label{lemma:1LMMid}
Assume that
\begin{equation}
\label{eq:1LMM}
\sum\nolimits_{i = 1}^s {d_i [1 + \exp ( - (x - \mu _i )/\sigma _i )]^{ - 1} }  \equiv 0,
\end{equation}
if $( \mu _i, \sigma _i )$, $i = 1,...,s$ are distinct and $\sigma _i  > 0$, $i = 1,...,s$, then $d _i  = 0$, $i = 1,...,s$.
\end{lem}
\begin{proof}
We proof this lemma by contradiction. Assume that $d_i  \ne 0,i = 1,...,s$. Let $x \to \infty$ in Eqn.~(\ref{eq:1LMM}), we obtain $
\sum\limits_{i = 1}^s {d_i }  = 0$. Thus with Eqn.~(\ref{eq:1LMM}), we have
\begin{equation}
\label{eq:-1LMM}
\sum\nolimits_{i = 1}^s {d_i [1 + \exp ( (x - \mu _i )/\sigma _i)]^{ - 1} }  \equiv 0.
\end{equation}
If now $\sigma _1  = ... = \sigma _s$, then $\mu _i ,i = 1,...,s$ are distinct. Multiply both sides of Eqn.~(\ref{eq:-1LMM}) by $1 + \exp ((x - \mu _1 )/\sigma _1 )$ and take the derivative w.r.t. $x$, we obtain
\begin{equation}
\label{eq:-1LMMiterStart}
\begin{split}
&\sum\nolimits_{i = 2}^s {d_i [\exp ( - \mu _1 /\sigma _1 ) - \exp ( - \mu _i /\sigma _i )][1 + \exp ((x - \mu _i )/\sigma _i )]^{-2} } \\ &\equiv 0.
\end{split}
\end{equation}
Then iteratively multiply both sides of Eqn.~(\ref{eq:-1LMMiterStart}) by $1 + \exp ((x - \mu _i )/\sigma _i ),i = 2,...,s$ and derive respect to $x$ similar to the process applied to Eqn.~(\ref{eq:-1LMM}) above, we have $d_s (\prod\nolimits_{i = 1}^{s - 1} {[\exp (- \mu _i /\sigma _i ) - \exp ( - \mu _1 /\sigma _1 )]} )[1 + \exp (\frac{{x - \mu _s }}{{\sigma _s }})]^{-s}  \equiv 0$. Then we obtain $d_s  = 0$ which contradicts to the assumption.

On the other hand, without loss of generality, assume $\sigma _1  = \sigma _2  = ... = \sigma _k  < \sigma _{k + 1}  \le ... \le \sigma _s
$.  Multiply both sides of Eqn.~(\ref{eq:-1LMM}) by $\exp (x/\sigma)$, where $\sigma  < \sigma _1$, and let $x \to \infty$, we have $\sum\limits_{i = 1}^k {d_i \exp (\mu _i /\sigma _i ) = 0}$. Derive both sides of Eqn.~(\ref{eq:-1LMM}) with respect to $x$, we have
 \begin{equation}
 \label{eq:1derive}
\begin{split}
\sum\nolimits_{i = 1}^s ({d_i /\sigma _i [1 + \exp ((x - \mu _i )/\sigma _i )]^{ - 1} }
\\-  {d_i /\sigma _i [1 + \exp ((x - \mu _i )/\sigma _i )]^{ - 2} })  \equiv 0.
\end{split}
\end{equation}
Multiply both sides by $\exp (x/\sigma)$, where $\sigma  < \sigma _1 /2$, and let $x \to \infty$, we have $\sum\limits_{i = 1}^k {d_i \exp (2\mu _i /\sigma _i ) = 0}$. Then similarly, derive  \label{eq:1derive} $j$ times with respect to $x$, and multiply by $\exp (x/\sigma)$, where $\sigma  < \sigma _1 /j, j=2,...,k$. Then let $x \to \infty$, we obtain $\sum\limits_{i = 1}^k {d_i \exp (j\mu _i /\sigma _i ) = 0}$. Thus we obtain
 \begin{equation}
\left\{ \begin{array}{l}
 \sum\limits_{i = 1}^k {d_i \exp (\mu _i /\sigma _i ) = 0}  \\
 \sum\limits_{i = 1}^k {d_i \exp (2\mu _i /\sigma _i ) = 0}  \\
 ... \\
 \sum\limits_{i = 1}^k {d_i \exp (k\mu _i /\sigma _i ) = 0}  \\
 \end{array} \right.
\end{equation}
From Vandermonde's rule, we have $d _i  = 0$, $i = 1,...,k$ which contradict to the hypothesis. Thus the lemma is proved.
\end{proof}

\begin{lem}
\label{lemma:mLMMid}
Assume that
\begin{equation}
\label{eq:3LMM}
\sum\nolimits_{i = 1}^s {d_i [1 + \sum\nolimits_{k = 1}^p {\exp ( - (x_k  - \mu _i^k )/\sigma^k )} ]^{ - 1} }  \equiv 0,
\end{equation}
if $ \mathbf{\mu} _i $, $i = 1,...,s$ are distinct, where $\mathbf{\mu} _i  = (\mu _i^1 ,...,\mu _i^p )$, $\mu _i  \in R^p, \mathbf{\Sigma}  = (\sigma^1 ,...,\sigma^p )$, and $\mathbf{\Sigma} \in (0,\infty )^p$ then $d _i  = 0$, $i = 1,...,s$.
\end{lem}
\begin{proof}
We give a proof by inductive on $s$ and $p$. When $s=1$, it is trivial and $p=1$, it is Lemma~\ref{lemma:1LMMid}. Setting $s > 1$ fixed, let $p > 1$ and assume $p-1$ is true.
Let $x_1  = \sigma^1 x - y_1 ,x_2  = \sigma^2 x - y_2$, then we can find $(y_1^0,y_2^0)$ satisfies that if $(\mu _i^1 ,\mu _i^2 ) \ne (\mu _j^1 ,\mu _j^2 )$, then $\exp ((y_1^0  + \mu _i^1)/\sigma^1) + \exp ((y_2^0  + \mu _i^2 )/\sigma^2) \ne \exp ((y_1^0  + \mu _j^1 )/\sigma^1 ) + \exp ((y_2^0  + \mu _j^2)/\sigma^2)$. Otherwise, for each $(y_1,y_2)$, there exists a pair $(i,j)$ that$( \mu _i^1 ,\mu _i^2 ) \ne (\mu _j^1 ,\mu _j^2 )$, and $\exp ((y_1  + \mu _i^1)/\sigma ^1) + \exp ((y_2  + \mu _i^2 )/\sigma^2) = \exp ((y_1  + \mu _j^1 )/\sigma^1 ) + \exp ((y_2  + \mu _j^2)/\sigma ^2)$.
Then there must exists a pair $(i,j)$ and a two dimensional series $(y_1^{h,k},y_2^{h})$ where $\mathop {\lim }\limits_{h \to \infty } y_2^h  = \infty$, and for fixed $h$, $\mathop {\lim }\limits_{k \to \infty } y_1^{h,k}  = \infty$, that $(\mu _i^1 ,\mu _i^2 ) \ne (\mu _j^1 ,\mu _j^2 )$, and
\begin{equation}
\label{eq:condition}
\begin{split}
\exp ((y_1^{h,k}  + \mu _i^1)/\sigma^1) + \exp ((y_2^h  + \mu _i^2 )/\sigma^2) \\
 = \exp ((y_1^{h,k}  + \mu _j^1 )/\sigma^1 ) + \exp ((y_2^h  + \mu _j^2)/\sigma^2).
 \end{split}
\end{equation}
Let $h$ fixed, and $k \to \infty$ in Eqn.~\ref{eq:condition}, we obtain $\mu _i^1=\mu _j^1$. Then let $h \to \infty$, we have $\mu _i^2=\mu _j^2$, which contradict to the assumptions.

So let $x_1  = \sigma^1 x - y_1^0 ,x_2  = \sigma^2 x - y_2^0$ in Eqn.~\ref{eq:3LMM}, then it is reduced to the situation $p-1$. Thus the lemma is proved.
\end{proof}

\textbf{Proof of Theorem~\ref{thm:main_results_1dim}}:
\begin{proof}
According to the theorem of~\cite{yakowitz1968identifiability}, Theorem~\ref{thm:main_results_1dim} is equivalent to the linear independence of $\mathfrak{L}$ over the field of real numbers which can be inferred immediately from Lemma~\ref{lemma:mLMMid}. Thus when $p=1$ the class $\mathfrak{L}$ with parameter set $\Theta$ and mixing distribution in $\mathfrak{R}$ as defined above is identifiable.
\end{proof}

\textbf{Proof of Theorem~\ref{thm:main_results_pdim}}:
\begin{proof}
Similar with the proof above, from Lemma~\ref{lemma:mLMMid} and the theorem of~\cite{yakowitz1968identifiability}, Theorem~\ref{thm:main_results_pdim} is obvious.
\end{proof}

\textbf{Open Problem}: When $p>1$, the class $\mathfrak{L}$ with parameter set $\mathbf{\Theta}$ and mixing distribution in $\mathfrak{R}$ is identifiable.

\section{Conclusions}
\label{sec:Conclusions}
In this work, we examine the identifiability property of logistic mixture models. We first proof that the univariate logistic mixture models is identifiable. Then we also proof the identifiability for a special case of the multivariate logistic mixture models. At last we propose a conjecture about identifiability of the multivariate logistic mixture models.


\vskip3pt
{\bf ACKNOWLEDGEMENTS.} This work was partly supported by the National Natural Science Foundation of China under grant No. 91120303 and No. 61071181.

\vskip5pt

\noindent Z.Q. Shi and T.R. Zheng and J.Q. Han 
\vskip3pt

\noindent E-mail: shiziqiang7@gmail.com; shiziqiang@cn.fujitsu.com

\end{document}